\documentclass[11pt]{article}
\usepackage{amssymb,amsmath}
\usepackage[pdftex]{color}
\usepackage{listings,graphicx,amsmath,varioref,amscd,amssymb,color,bm,stmaryrd,geometry,amsthm,amsfonts,graphics,bbm,lipsum,float}
\usepackage[english]{babel}
\usepackage[latin1]{inputenc}
\usepackage[T1]{fontenc}
\usepackage[colorlinks=true,citecolor=red,linkcolor=blue,urlcolor=RubineRed,pdfpagetransition=Blinds,pdftoolbar=false,pdfmenubar=false]{hyperref}

\newtheorem{theorem}{Theorem}
\newtheorem{proposition}{Proposition}

\newtheorem{lemma}{Lemma}

\newtheorem{remark}{Remark}

\newcommand{\R}{\mathbb{R}}
\newcommand{\Om}{\Omega}

\numberwithin{equation}{section}
\numberwithin{theorem}{section}
\numberwithin{proposition}{section}
\numberwithin{definition}{section}
\numberwithin{corollary}{section}
\numberwithin{lemma}{section}
\numberwithin{example}{section}

\begin{document}
\title{}
\begin{center}
\Large\textbf{On Some  Model Problem for the Propagation of Interacting Species in a Special Environment}\par \bigskip
\normalsize
\vskip .3 cm
\large Michel Chipot$^{(a)}$ \& Mingmin Zhang$^{(b)}$\par \bigskip
\small $^{(a)}$Institut f\"ur Mathematik \\Universit\"at Z\"urich \\Winterthurerstrasse 190 \\8057 Z\"urich, Switzerland\\\texttt{m.m.chipot@math.uzh.ch} \par \bigskip
	$^{(b)}$School of Mathematical Sciences\\
	 University of Science and Technology of China\\ 
	 Hefei, Anhui 230026, China\\
	 and\\
	 	Aix Marseille Universit\'e\\ CNRS, Centrale Marseille, I2M\\ UMR 7373, 13453
	 Marseille, France\\
 \texttt{lang925@mail.ustc.edu.cn}
 \par 
 \end{center}
\vskip .3 cm
\begin{abstract}The purpose of this note is  to  study  the existence of a nontrivial solution for an elliptic system which comes from a newly introduced mathematical problem so called Field-Road model. Specifically, it consists of coupled equations set in domains of different dimensions together with some interaction of  non classical type. We consider a truncated problem  by imposing  Dirichlet boundary conditions and  an unbounded setting as well.
 \end{abstract}


\vskip .3 cm


\section{Introduction and notation} 
It has been observed in diffusion or propagation problems that the topography of the environment is playing a crucial r\^ole. We refer for instance to \cite{B1}, \cite{B} and the references there. In particular the roads are not only essential for human beings but also can be useful for other species. One can easily imagine that mosquitos could be interested in socialising with human beings along a road but many other situations seem to have occurred due to this element of civilisation. In this note we consider a model problem where the living space of our species consists in a field and one or several roads that we will assume to be unidimensional.  
We will also assume the roads to be straight but several extensions can be addressed very easily with our approach. On each domain -field or road- we will consider nonlinear diffusion equations which include for instance the Fisher-KPP types. We restrict ourselves to two species. The interaction with the two species is  modelled by a particular flux condition (see \cite{B1}, \cite{B}  and the set of equation below). 
\vskip .3 cm
Let $\Om_\ell$ be the open set of $\R^2$, defined for $\ell, L >0$ as 
\begin{equation*}
\Om_\ell = (-\ell, \ell) \times (0,L).
\end{equation*}
We denote by $\Gamma_0$ the part of the boundary of $\Om_\ell$  located on the $x_1$-axis i.e.
\begin{equation*}
\Gamma_0 = (-\ell, \ell) \times \{0\}
\end{equation*}
and by $\Gamma_1$ the rest of the boundary that is to say 
\begin{equation*}
\Gamma_1 = \partial \Om_\ell \backslash \Gamma_0.
\end{equation*}
When convenient we will identify $\Gamma_0$ to $(-\ell, \ell)$. In this setting $\Om_\ell$ stands for a field and $\Gamma_0$ for a portion of a road.
\begin{figure}[H]
	\centering
	\includegraphics[scale=1]{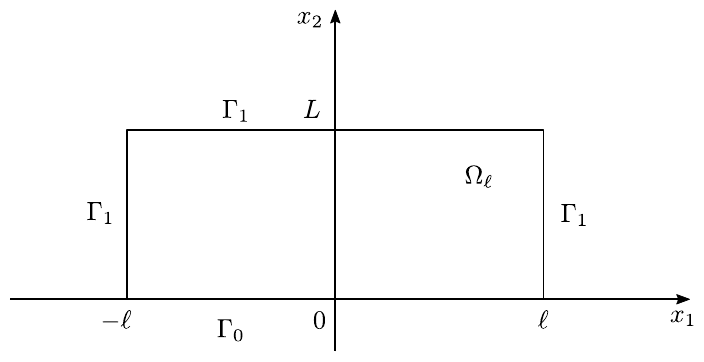}
	\caption{The domain $\Omega_\ell$ for one-road problem}
\end{figure}
\noindent Set 
\begin{equation*}
V = \{v\in H^1( \Om_\ell) ~|~ v=0 \text{ on } \Gamma_1 \}.
\end{equation*}
We would like to find a solution to the problem
\begin{equation*}
\begin{cases}
-D\Delta v = f(v) \text{ in }  \Om_\ell, \cr
v=0 \text{ on } \Gamma_1, ~~D \frac{\partial v}{\partial n} = \mu u-\nu v \text{ on } \Gamma_0, \cr
-D'u'' + \mu u = g(u) + \nu v \text{ in } \Gamma_0, \cr
u=0 \text{ on } \partial \Gamma_0 =\{-\ell,\ell\}.\cr
\end{cases} 
\end{equation*}
($n$ denotes the outward unit normal to $ \Om_\ell$). 
\vskip .3 cm
\noindent In the weak form we would like to find a couple $(u,v)$ such that

\begin{equation}\label{1.1}
\begin{cases}
(u,v) \in H^1_0(\Gamma_0)\times V, \cr

\int_{\Om_\ell} D\nabla v\cdot \nabla \varphi ~dx +  \int_{\Gamma_0} \nu v(x_1,0)\varphi  ~ dx_1 =  \int_{\Om_\ell} f(v)\varphi ~dx + \int_{\Gamma_0} \mu u \varphi ~dx_1 ~~\forall \varphi \in V, \cr

\int_{\Gamma_0} D'u'\psi'  +\mu u \psi~dx_1 = \int_{\Gamma_0} \nu v(x_1,0)\psi ~dx_1 +  \int_{\Gamma_0} g(u) \psi ~dx_1 ~~\forall \psi  \in H^1_0(\Gamma_0). \cr
\end{cases} 
\end{equation}

Here we assume that 
\begin{equation*} D, D',m, \mu,\nu ~~\text{are positive constants,}
\end{equation*}
such that
\begin{equation}\label{1.2}
m\geq \frac{\nu}{\mu}.
\end{equation}

$f, g$ are Lipschitz continuous functions i.e. such that for some positive constants $L_f, L_g$ it holds 
\begin{equation} \label{1.3}
|f(x)- f(y) | \leq L_f |x-y| , ~~|g(x)- g(y) | \leq L_g |x-y|~~\forall x,y \in \R.
\end{equation}
Note that this implies that for $\lambda \geq L_f$ (respectively  $\eta \geq L_g$) the functions 
\begin{equation} \label{1.4}
x \to \lambda x - f(x),~ \eta x - g(x)
\end{equation}
are nondecreasing. In addition we will assume 
\begin{equation} \label{1.5}
f(0)=f(1) =0, ~~f > 0 \text{ on } (0,1), ~~f \leq 0 \text{ on } (1, +\infty).
\end{equation}
\begin{equation} \label{1.6}
g(0)=0, ~~g(m) \leq 0.
\end{equation}
One should remark that under the assumptions above, $(0,0)$ is a solution to \eqref{1.1}. We are interested in finding a nontrivial solution to the problem \eqref{1.1}. For the notation and the usual properties on Sobolev spaces we refer to \cite{Ci}, \cite{C}, \cite{DL}, \cite{E}, \cite{GT}. We will also consider some extension in the case of a  two-road problem which consists of three coupled equations with two interaction conditions on the upper- and lower- boundaries.  Finally, we will address the case of an unbounded setting for the one-road problem.

\section{Preliminary results}

\begin{lemma}\label{lemma-2.1} Suppose that $w $ is a measurable function on ${\Gamma_0}$ such that 
\begin{equation*}\label{2.1}
0\leq w \leq m.
\end{equation*}
Then under the assumptions above  the problem 
\begin{equation}\label{2.2}
\begin{cases}
v\in V, \cr
\int_{\Om_\ell} D\nabla v\cdot \nabla \varphi ~dx~ +&\int_{\Gamma_0} \nu  v(x_1,0)\varphi (x_1,0) ~ dx_1 \cr &=  \int_{\Om_\ell} f(v)\varphi ~dx + \int_{\Gamma_0} \mu w \varphi ~dx_1 ~~\forall \varphi \in V, \cr
\end{cases} 
\end{equation}
possesses a minimal and a maximal solution with values in $(0, \frac{\mu}{\nu}m)$.
\end{lemma}

\begin{proof} Let us remark first that any nonnegative solution to \eqref{2.2} takes its values in $(0, \frac{\mu}{\nu}m)$. Indeed if $v$ is solution to \eqref{2.2} taking as test function $\varphi = (v-k)^+$, $k= \frac{\mu}{\nu}m\geq 1$ one gets 
\begin{equation*}
\begin{aligned}
\int_{\Om_\ell} D|\nabla (v-k)^+|^2 ~dx &= \int_{\Om_\ell} D\nabla (v-k)\cdot \nabla (v-k)^+~dx = \int_{\Om_\ell} D\nabla v\cdot \nabla (v-k)^+~dx \cr
&=   \int_{\Om_\ell} f(v)(v-k)^+~dx + \int_{\Gamma_0} \{\mu w - \nu v(x_1,0)\}(v-k)^+ ~ dx_1 \cr
 &  \leq  \int_{\Gamma_0} \{\mu w - \nu v \}(v-k)^+ ~ dx_1 \leq 0, \cr
\end{aligned}
\end{equation*}
since on the set where $v \geq k$ one has $v\geq \frac{\mu}{\nu} m$ and $\{\mu w - \nu v\}\leq \{\mu w - \mu m\} \leq0$.
\vskip .3 cm
Next let us note that $0$ is a subsolution to \eqref{2.2}. Indeed this follows trivially from 
\begin{equation*}\label{2.3}
\int_{\Om_\ell} D\nabla 0\cdot \nabla \varphi ~dx +  \int_{\Gamma_0} \nu 0\varphi ~ dx_1 \leq  \int_{\Om_\ell} f(0)\varphi ~dx + \int_{\Gamma_0} \mu w \varphi ~dx_1 ~~\forall \varphi \in V,  \varphi  \geq 0.
\end{equation*}
On the other hand, for $\varphi \in V$, $\varphi \geq 0$,  one has also for $k= \frac{\mu}{\nu}m$, since $k\geq 1$
\begin{equation*}\label{2.4}
\int_{\Om_\ell} D \nabla k\cdot \nabla \varphi ~dx + \int_{\Gamma_0} \nu k \varphi ~dx_1 \geq \int_{\Om_\ell} f( k) \varphi ~dx + \int_{\Gamma_0}\mu w \varphi ~dx_1
\end{equation*}
and the function constant equal to $k= \frac{\mu}{\nu}m$ is a supersolution to \eqref{2.2}.
\vskip .3 cm
For $z\in L^2(\Om_\ell)$ we denote by $y=S(z)$ the solution to 

\begin{equation}\label{2.5}
\begin{cases}
y\in V, \cr
\int_{\Om_\ell} D\nabla y\cdot \nabla \varphi ~dx & + \int_{\Om_\ell} \lambda y \varphi ~dx + \int_{\Gamma_0} \nu y(x_1,0)\varphi  ~ dx_1 \cr &=  \int_{\Om_\ell} f(z)\varphi ~dx + \int_{\Om_\ell} \lambda z \varphi ~dx  +\int_{\Gamma_0} \mu w \varphi ~dx_1 ~~\forall \varphi \in V, \cr
\end{cases} 
\end{equation}
where we have chosen $\lambda \geq L_f$. Note that the existence of a unique solution $y$ to the problem above is an immediate consequence of the Lax-Milgram theorem. 
\vskip .3 cm
First we claim that the mapping $S$ is continuous from $L^2(\Om_\ell)$ into itself. Indeed setting $y'=S(z')$ one has by subtraction of the equations satisfied by $y$ and $y'$ 
\begin{equation}\label{2.6}
\begin{aligned}
\int_{\Om_\ell} D\nabla (y-y')&\cdot \nabla \varphi ~dx + \int_{\Om_\ell} \lambda (y-y') \varphi ~dx + \int_{\Gamma_0} \nu (y-y')\varphi  ~ dx_1  \cr 
&=  \int_{\Om_\ell} \{f(z)-f(z')\}\varphi ~dx + \int_{\Om_\ell} \lambda (z-z') \varphi ~dx  ~~\forall \varphi \in V. 
\end{aligned}
\end{equation}

Taking $\varphi = y-y'$ one derives easily 
\begin{equation*}
\begin{aligned}
 \lambda   \int_{\Om_\ell}|y-y'|^2~dx &\leq   \int_{\Om_\ell} |f(z)-f(z')||y-y'| ~dx +  \lambda  \int_{\Om_\ell}|z-z'||y-y'| ~dx \cr
& \leq ( L_f +  \lambda ) \int_{\Om_\ell}|z-z'||y-y'| ~dx. \cr
\end{aligned}
\end{equation*}
By the Cauchy-Schwarz inequality we obtain then
\begin{equation*}
|S(z)-S(z')|_{2,{\Om_\ell}} \leq \frac{L_f+\lambda}{\lambda}|z-z'|_{2,{\Om_\ell}}
\end{equation*}
which shows the continuity of the mapping $S$ ($|~~|_{2,{\Om_\ell}}$ denotes the usual $L^2({\Om_\ell})$-norm).
\vskip .3 cm
We show now that the mapping $S$ is monotone. Indeed suppose that $z\geq z'$ and as above denote by $y'$ the function $S(z')$. Taking $\varphi = -(y-y')^-$ in \eqref{2.6} we get 
\begin{equation*}
\begin{aligned}
\int_{\Om_\ell} D|\nabla (y-y')^-|^2~dx &+ \int_{\Om_\ell} \lambda ((y-y')^-)^2~dx + \int_{\Gamma_0} \nu((y-y')^-)^2  ~ dx_1  \cr 
&=  - \int_{\Om_\ell} [ \lambda (z-z') +  \{f(z)-f(z')\}] (y-y')^- ~dx \leq 0, 
\end{aligned}
\end{equation*}
since $- \{f(z)-f(z') \}\leq L_f |z-z'| \leq \lambda (z-z')$ (see \eqref{1.3}, \eqref{1.4}). This shows that $(y-y')^- = 0$ and the monotonicity of the mapping $S$. 
\vskip .3 cm
We consider now the following sequences  (Cf. \cite{A}) :
\begin{equation}\label{2.7}
\begin{aligned}
&{\underline{y_0}}= 0,~~~~ {\overline{y_0}}= \frac{\mu}{\nu}m =k, \cr
&{\underline{y_n}}= S({\underline{y_{n-1}}}), ~~~~{\overline{y_n}}= S({\overline{y_{n-1}}}) , ~~n\geq 1.
\end{aligned}
\end{equation}
One has 
\begin{equation}\label{2.8}
{\underline{y_0}}= 0 \leq {\underline{y_1}} \leq \cdots \leq {\underline{y_n}} \leq {\overline{y_n}} \leq\cdots  \leq {\overline{y_1}}  \leq {\overline{y_0}}= \frac{\mu}{\nu}m=k.
\end{equation}
Indeed since ${\underline{y_1}} = S({\underline{y_0}})=S(0)$ one has for $\forall \varphi \in V$, $ \varphi \geq 0$, 
\begin{equation*}\label{2.9}
\begin{aligned}
\int_{\Om_\ell} D\nabla \underline{y_1}\cdot \nabla \varphi ~dx  & + \int_{\Om_\ell} \lambda \underline{y_1} \varphi ~dx + \int_{\Gamma_0} \nu \underline{y_1}(x_1,0)\varphi  ~ dx_1\cr&=  \int_{\Om_\ell} f(\underline{y_0})\varphi ~dx + \int_{\Om_\ell} \lambda \underline{y_0} \varphi ~dx  +\int_{\Gamma_0} \mu w \varphi ~dx_1 \cr
& \geq \int_{\Om_\ell} D\nabla \underline{y_0}\cdot \nabla \varphi ~dx   + \int_{\Om_\ell} \lambda \underline{y_0} \varphi ~dx + \int_{\Gamma_0} \nu \underline{y_0}(x_1,0)\varphi  ~ dx_1, 
\end{aligned}
\end{equation*}
since $\underline{y_0}$ is a subsolution to \eqref{2.2}. Using this inequality with $ \varphi = (\underline{y_1} -\underline{y_0})^-$
one derives easily 
\begin{equation*}
\int_{\Om_\ell} D\nabla (\underline{y_1}-\underline{y_0})\cdot \nabla (\underline{y_1} -\underline{y_0})^-~dx   + \int_{\Om_\ell} \lambda(\underline{y_1}-\underline{y_0})(\underline{y_1} -\underline{y_0})^- ~dx + \int_{\Gamma_0} \nu (\underline{y_1}-\underline{y_0})(\underline{y_1} -\underline{y_0})^- ~ dx_1\geq 0. 
\end{equation*}
Thus it follows that $\underline{y_1}\geq\underline{y_0}$. With a similar proof one gets that  $\overline{y_1}\leq\overline{y_0}$.
Applying $S^{n-1}$ to these inequalilties leads to 
\begin{equation*}
S^{n-1}(\underline{y_0}) = \underline{y_{n-1}} \leq  \underline{y_n}\ = S^{n-1}(\underline{y_1})~~, ~~S^{n-1}(\overline{y_0}) = \overline{y_{n-1} }\geq \overline{y_{n}} =S^{n-1}(\overline{y_1}).
\end{equation*}
Furthermore from $ \underline{y_0} \leq \overline{y_0}$ one derives by applying $S^n$ to both sides of the inequality 
\begin{equation*}
\underline{y_n} \leq  \overline{y_{n}}.
\end{equation*}
This completes the proof of \eqref{2.8}. Then for some functions $\underline{v},\overline{v}$ in $L^2(\Om_\ell)$ one has 
\begin{equation*}
\underline{y_n} \to \underline{v},~~   \overline{y_{n}} \to  \overline{v}~~\text{ in } L^2(\Om_\ell).
\end{equation*}
Clearly $ \underline{v}$ and $ \overline{v}$ are fixed point for $S$ and thus (see \eqref{2.5}) solutions to \eqref{2.2}. This completes the proof of the lemma.
\end{proof}

\vskip .3 cm

We denote by $\lambda_1= \lambda_1(\Om_\ell)$ the first eigenvalue of the Dirichlet problem in $\Om_\ell$ and by $\varphi_1$ the corresponding first eigenfunction  positive and normalised. More precisely $(\lambda_1, \varphi_1)$ is such that
\begin{equation}\label{2.10}
\begin{cases} -\Delta \varphi_1 = \lambda_1 \varphi_1 \text{ in } \Om_\ell, \cr
\varphi_1 = 0 \text{ on } \partial \Om_\ell, \cr
\varphi_1  (0,\frac{L}{2}) =1.
\end{cases}.
\end{equation}

We suppose that for $s>0$ small enough one has 
\begin{equation}\label{2.11}
\lambda_1 \leq \frac{f(s)}{Ds}.
\end{equation}

Then one has :
\begin{lemma}
	\label{lemma-2.2}
	 Under the assumptions of the preceding lemma and \eqref{2.11}, for $\epsilon>0$ small enough, the maximal solution $\overline v$ to \eqref{2.2} satisfies 
\begin{equation*}\label{2.12}
\epsilon \varphi_1  \leq \overline v.
\end{equation*}
 In particular $\overline v$ is bounded away from $0$.
\end{lemma}

\begin{proof}
Due to \eqref{2.11} one has for $\epsilon>0$ small enough
\begin{equation*}
D\lambda_1 \epsilon \varphi_1 \leq f( \epsilon \varphi_1).
\end{equation*}
This allows us to show that $ \epsilon \varphi_1$ is a subsolution to \eqref{2.2}. 
Indeed, for $\varphi \in V$, $\varphi \geq 0$ it holds after integration by parts
\begin{equation*}
\begin{aligned}
\int_{\Om_\ell} D \nabla( \epsilon \varphi_1)\cdot \nabla \varphi ~dx &+ \int_{\Gamma_0} \nu \epsilon \varphi_1(x_1,0) \varphi ~dx_1  \cr
&=   \int_{\Om_\ell} D \nabla( \epsilon \varphi_1)\cdot \nabla \varphi ~dx =  \int_{\Om_\ell}  \nabla \cdot (D \nabla (\epsilon \varphi_1) \varphi )- D\Delta (\epsilon \varphi_1) \varphi  ~dx\cr
 &  = \int_{\partial \Om_\ell}D\partial_{n} (\epsilon \varphi_1)  \varphi d\sigma +  \int_{\Om_\ell}  D \lambda_1 (\epsilon \varphi_1) \varphi  ~dx\cr
 &\leq  \int_{\Om_\ell}  f(\epsilon \varphi_1) \varphi  ~dx + \int_{\Gamma_0} \mu w \varphi ~dx_1. \cr
\end{aligned}
\end{equation*}
($n$ denotes the outward unit normal to $\Om_\ell$, note that $\partial_{n} (\epsilon \varphi_1) \leq 0$). Thus $\epsilon \varphi_1$ is a positive subsolution to \eqref{2.2}.
\vskip .3 cm
Then one argues as in the preceding lemma introducing the sequence 
defined for $\epsilon$ small by :
\begin{equation*}
\begin{aligned}
&{\underline{y_0}}= \epsilon \varphi_1 \leq \frac{\mu}{\nu}m,~~~~ {\overline{y_0}}= \frac{\mu}{\nu}m, \cr
&{\underline{y_n}}= S({\underline{y_{n-1}}}), ~~~~{\overline{y_n}}= S({\overline{y_{n-1}}}) , ~~n\geq 1.
\end{aligned}
\end{equation*}
One has with the same proof as above 
\begin{equation*}
{\underline{y_0}}= \epsilon \varphi_1 \leq {\underline{y_1}} \leq \cdots \leq {\underline{y_n}} \leq {\overline{y_n}} \leq\cdots  \leq {\overline{y_1}}  \leq {\overline{y_0}}= \frac{\mu}{\nu}m.
\end{equation*}
The result follows from the fact that $\overline{y_n} \to \overline{v}$. This completes the proof of the lemma.
\end{proof}

One has also :
\begin{lemma}\label{lemma-2.3} Suppose that 
\begin{equation}\label{2.13}
\frac{f(s)}{s} \text{ is decreasing on } (0,+\infty).
\end{equation}
If $v_1, v_2$ are positive solutions to \eqref{2.2} corresponding to $w_1,w_2$ respectively then 
\begin{equation*}\label{2.14}
w_1 \leq w_2 \text{ implies }  v_1 \leq v_2.
\end{equation*}
In particular \eqref{2.2} has a unique positive solution. 
\end{lemma}
\begin{proof} Denote by $\theta$ a smooth function such that 
\begin{equation*}
\theta(t) =0 ~~\forall  t\leq 0,~~ \theta(t) =1 ~~ \forall  t\geq 1, ~~\theta'(t) \geq 0.
\end{equation*}
Set $\theta_\epsilon(t)= \theta(\frac{t}{\epsilon})$. Clearly 
\begin{equation*}
v_1\theta_\epsilon(v_1-v_2),~v_2\theta_\epsilon(v_1-v_2) \in V.
\end{equation*}

From the equations satisfied by $v_1,v_2$ one gets, setting $\theta_\epsilon =\theta_\epsilon(v_1-v_2) $,
\begin{equation*}
\int_{\Om_\ell} D\nabla v_1\cdot \nabla (v_2\theta_\epsilon) ~dx +  \int_{\Gamma_0} \nu v_1(x_1,0)(v_2\theta_\epsilon) ~ dx_1 =  \int_{\Om_\ell} f(v_1)(v_2\theta_\epsilon) ~dx + \int_{\Gamma_0} \mu w_1(v_2\theta_\epsilon) ~dx_1,
\end{equation*}
\begin{equation*}
\int_{\Om_\ell} D\nabla v_2\cdot \nabla (v_1\theta_\epsilon) ~dx +  \int_{\Gamma_0} \nu v_2(x_1,0)(v_1\theta_\epsilon) ~ dx_1 =  \int_{\Om_\ell} f(v_2)(v_1\theta_\epsilon) ~dx + \int_{\Gamma_0} \mu w_2(v_1\theta_\epsilon) ~dx_1.
\end{equation*}
By subtraction we obtain
\begin{equation*}
\begin{aligned}
\int_{\Om_\ell} D\{\nabla v_2\cdot \nabla (v_1\theta_\epsilon) -\nabla v_1\cdot \nabla (v_2\theta_\epsilon)\}~dx  =  \int_{\Om_\ell} &f(v_2)(v_1\theta_\epsilon)- f(v_1)(v_2\theta_\epsilon) ~dx\cr & + \int_{\Gamma_0} \mu (w_2v_1-w_1v_2)\theta_\epsilon (v_1-v_2)~dx_1.
\end{aligned}
\end{equation*}
Clearly the last integral above is nonnegative so that one has 
\begin{equation*}
\int_{\Om_\ell} D\{\nabla v_2\cdot \nabla (v_1\theta_\epsilon) -\nabla v_1\cdot \nabla (v_2\theta_\epsilon)\}~dx  \geq  \int_{\Om_\ell} f(v_2)(v_1\theta_\epsilon)- f(v_1)(v_2\theta_\epsilon) ~dx .
 \end{equation*}
By a simple computation writing $\theta'_\epsilon $ for $ \theta'_\epsilon(v_1-v_2) $ one derives 
\begin{equation*}
\begin{aligned}
\int_{\Om_\ell} f(v_2)(v_1\theta_\epsilon)&- f(v_1)(v_2\theta_\epsilon) ~dx  \leq \int_{\Om_\ell} D\{\nabla v_2\cdot \nabla (v_1\theta_\epsilon) -\nabla v_1\cdot \nabla (v_2\theta_\epsilon)\}~dx  \cr
& =  \int_{\Om_\ell} D\{\nabla v_2\cdot \nabla (v_1 -v_2)\theta'_\epsilon v_1 - \nabla v_1\cdot \nabla (v_1 -v_2)\theta'_\epsilon v_2\}~dx \cr
&= \int_{\Om_\ell} D \{v_1 \nabla v_2 -v_2 \nabla v_1\}\cdot\nabla (v_1 -v_2)\theta'_\epsilon~dx \cr
&= \int_{\Om_\ell} D \{v_1 \nabla v_2 -v_2 \nabla v_2+v_2 \nabla v_2-v_2 \nabla v_1\}\cdot\nabla (v_1 -v_2)\theta'_\epsilon~dx \cr
&=\int_{\Om_\ell} D \nabla v_2\cdot\nabla (v_1 -v_2)(v_1 -v_2) \theta'_\epsilon ~dx -\int_{\Om_\ell} D v_2|\nabla (v_1 -v_2)|^2\theta'_\epsilon~dx \cr 
&\leq \int_{\Om_\ell} D \nabla v_2\cdot\nabla (v_1 -v_2)(v_1 -v_2) \theta'_\epsilon ~dx.\cr
\end{aligned}
 \end{equation*}
 Let us set $\gamma_\epsilon(t) = \int_0^t s\theta'_\epsilon(s)ds$ in such a way that the inequality above reads 
 \begin{equation*}
\int_{\Om_\ell} f(v_2)(v_1\theta_\epsilon)- f(v_1)(v_2\theta_\epsilon) ~dx  \leq  \int_{\Om_\ell} D \nabla v_2\cdot\nabla \gamma_\epsilon(v_1 -v_2) ~dx.
 \end{equation*}
From the equation satisfied by $v_2 $, since $\gamma_\epsilon(v_1 -v_2) \in V$ and $v_2, \gamma_\epsilon$ are nonnegative one has 
 \begin{equation*}
\int_{\Om_\ell}  D \nabla v_2\cdot\nabla \gamma_\epsilon(v_1 -v_2) ~dx \leq \int_{\Om_\ell} f(v_2)  \gamma_\epsilon(v_1 -v_2) ~dx + \int_{\Gamma_0} \mu w_2  \gamma_\epsilon(v_1 -v_2) ~dx_1 .
 \end{equation*}
 Since for some constant $C$
  \begin{equation*}
\gamma_\epsilon(t) \leq \int_0^\epsilon s\theta'(\frac{s}{\epsilon})\frac{1}{\epsilon}ds\leq  C\epsilon
\end{equation*}
the right hand side of the two inequalities above goes to $0$ when $\epsilon \to 0$. Since when $\epsilon \to 0$ one has $\theta_\epsilon(v_1-v_2) \to \chi_{\{v_1>v_2\}}$ the characteristic function of the set 
$\{v_1>v_2\} = \{ x\in \Om_\ell~ |~v_1(x)>v_2(x) \}$ one gets 
 \begin{equation*}
\int_{\{v_1>v_2\}}  f(v_2)v_1- f(v_1)v_2 ~dx  \leq 0.
 \end{equation*}
But on the set of integration thanks to \eqref{2.13} one has $f(v_2)v_1- f(v_1)v_2 >0$ hence the set of integration is necessarily of measure $0$, i.e. $v_1 \leq v_2$. This completes the proof of the lemma.
\end{proof}

\section{The main result} 

\begin{theorem}
\label{thm-3.1}	
 Suppose that \eqref{1.2}-\eqref{1.6},\eqref{2.11},\eqref{2.13} hold, then the problem \eqref{1.1} admits a nontrivial solution.
\end{theorem}
\begin{proof} As mentioned above it is of course clear that $(0,0)$ is solution to \eqref{1.1}. Set 
 \begin{equation*}
K = \{v \in L^2({\Gamma_0}) ~|~ 0\leq v\leq m\}.
 \end{equation*}
 For $u \in K$, let $\overline v$ be the unique positive solution to \eqref{2.2} associated to $w=u$. For $\eta \geq L_g$ let $U=T(u)$ the solution to 
 \begin{equation}\label{3.1}
\begin{cases}
U\in H^1_0(\Gamma_0), \cr
\int_{\Gamma_0} D'U'\psi'  +\mu U \psi + \eta U \psi~dx_1 = \int_{\Gamma_0} \nu \overline{v}(x_1,0)\psi +  g(u) \psi + \eta u \psi~dx_1  ~~\forall \psi  \in H^1_0(\Gamma_0). \cr
\end{cases} 
\end{equation}
The existence of $U$ is a consequence of the Lax-Milgram theorem.
\vskip .3 cm
We claim that $T$ is continuous on $K\subset L^2({\Gamma_0})$. Indeed suppose that $u_n \to u$ in $K$. Denote by $\overline{v}_n$ the solution to \eqref{2.2} associated to $u_n$ i.e. satisfying 
 \begin{equation}\label{3.2}
\int_{\Om_\ell} D\nabla\overline{v}_n\cdot \nabla \varphi ~dx +  \int_{\Gamma_0} \nu\overline{v}_n(x_1,0)\varphi  ~ dx_1 =  \int_{\Om_\ell} f(\overline{v}_n)\varphi ~dx + \int_{\Gamma_0} \mu u_n \varphi ~dx_1 ~~\forall \varphi \in V.
\end{equation}
Since $\overline{v}_n$ and $u_n$ are bounded, taking $\varphi= \overline{v}_n$ in the equality above one gets easily 
\begin{equation*}
\int_{\Om_\ell} D|\nabla\overline{v}_n|^2~dx +  \int_{\Gamma_0} \nu\overline{v}_n(x_1,0)^2  ~ dx_1 \leq C,
\end{equation*}
where $C$ is a constant independent of $n$. Thus, up to a subsequence, there exists $v \in V$ such that 
\begin{equation*}\label{3.3}
\overline{v}_n \rightharpoonup v \text{ in } H^1({\Om_\ell}), ~~\overline{v}_n \to v \text{ in } L^2({\Om_\ell}),~~\overline{v}_n(.\,,0)  \to v(.\,,0) \text{ in } L^2({\Gamma_0}).
\end{equation*}
Passing to the limit in \eqref{3.2} it follows from Lemma \ref{lemma-2.3}  that $v= \overline{v}$ the solution to \eqref {2.2} corres\-ponding to $w=u$. By uniqueness of the limit one 
has convergence of the whole sequence and in particular 
\begin{equation*}
\overline{v}_n(.\,,0)  \to \overline{v}(.\,,0) \text{ in } L^2({\Gamma_0}).
\end{equation*}
Passing to the limit in  \eqref{3.1} written for $u=u_n$ one derives $T(u_n) \to T(u)$ in $ L^2({\Gamma_0})$. 
\vskip .3 cm
We can show also that $T$ is monotone. Indeed, suppose that $u_1 \geq u_2$ and set $U_i=T(u_i)$, $i=1,2$. One has 
  \begin{equation*}
  \begin{aligned}
\int_{\Gamma_0} D'U_1'\psi'  +\mu U_1 \psi + \eta U_1 \psi~dx_1 = \int_{\Gamma_0} \nu \overline{v}_1(x_1,0)\psi +  g(u_1) \psi + \eta u_1 \psi~dx_1  ~~\forall \psi  \in H^1_0(\Gamma_0), \cr
\int_{\Gamma_0} D'U_2'\psi'  +\mu U_2 \psi + \eta U_2 \psi~dx_1 = \int_{\Gamma_0} \nu \overline{v}_2(x_1,0)\psi +  g(u_2) \psi + \eta u_2 \psi~dx_1  ~~\forall \psi  \in H^1_0(\Gamma_0). \cr
 \end{aligned}
\end{equation*}
By subtraction it comes 
  \begin{equation*}
  \begin{aligned}
\int_{\Gamma_0}& D'(U_1-U_2)'\psi'  +\mu (U_1-U_2) \psi + \eta (U_1-U_2) \psi~dx_1 \cr &= \int_{\Gamma_0} \nu (\overline{v}_1(x_1,0)-\overline{v}_2(x_1,0))\psi +  (g(u_1)-g(u_2)) \psi + \eta (u_1-u_2) \psi~dx_1  ~~\forall \psi  \in H^1_0(\Gamma_0). \cr
\end{aligned}
\end{equation*}
Choosing $\psi= -(U_1-U_2)^-$ and taking into account that, by Lemma \ref{lemma-2.3}, $\overline{v}_1(x_1,0)-\overline{v}_2(x_1,0)\geq 0$ and that for $\eta \geq L_g$, $(g(u_1)-g(u_2)) + \eta (u_1-u_2) \geq 0$ (Cf. \eqref{1.3},
\eqref{1.4}), one gets
 \begin{equation*}
\int_{\Gamma_0} D'|\{(U_1-U_2)^-\}'|^2  +\mu \{(U_1-U_2)^-\}^2 + \eta \{(U_1-U_2)^-\}^2~dx_1 \leq 0.
\end{equation*}
Thus $(U_1-U_2)^-=0$ and $ T(u_1) \geq T(u_2)$.
\vskip .3 cm
Next we assert that $T$ maps $K$ into itself. Indeed, if $U_0=T(0)$ one has, with an obvious notation for $\overline{v}_0$
\begin{equation*}
\int_{\Gamma_0} D'U_0'\psi'  +\mu U_0 \psi + \eta U_0 \psi~dx_1 = \int_{\Gamma_0} \nu \overline{v}_0(x_1,0)\psi ~dx_1  ~~\forall \psi  \in H^1_0(\Gamma_0).
\end{equation*}
Taking $\psi = -U_0^-$ one deduces easily since $\overline{v}_0 >0$ that $U_0 = T(0) \geq 0$.
Similarly if $U_m = T(m)$ one has, with an obvious notation for $\overline{v}_m$
\begin{equation*}
\int_{\Gamma_0} D'U_m'\psi'  +\mu U_m \psi + \eta U_m \psi~dx_1 = \int_{\Gamma_0} \nu  \overline{v}_m(x_1,0)\psi + g(m) \psi + \eta m \psi~dx_1  ~~\forall \psi  \in H^1_0(\Gamma_0).
\end{equation*}
Thus choosing $\psi = (U_m-m)^+$ it comes since $g(m) \leq 0$, $\overline{v}_m \leq \frac{\mu}{\nu}m$
\begin{equation*}
 \begin{aligned}
\int_{\Gamma_0} D'|\{(U_m-m)^+\}'|^2  &+(\mu  + \eta) \{(U_m-m)^+\}^2 ~dx_1  \cr&= \int_{\Gamma_0} (\nu \overline{v}_m  -\mu m) (U_m-m)^+dx_1 \leq 0.
 \end{aligned}
\end{equation*}
From which it follows that $U_m \leq m$. By the monotonicity of $T$ it results that $T$ maps the convex $K$ into itself. But clearly $T(K) \subset C^\frac{1}{2}(\Gamma_0)$  is relatively compact in 
$L^2(\Gamma_0)$. Thus, by the Schauder fixed point theorem (see \cite{E},  \cite{GT},  \cite{C}), $T$ has a fixed point in $K$ which leads to a nontrivial solution to \eqref{1.1}. This completes the proof of the theorem.
\end{proof}

If it is clear at this point that the solution we constructed is non degenerate in $v$ it is not clear that the same holds for $u$. In fact we have : 
\begin{proposition}
	\label{prop-3.1}
 Let $(u,v)$ be the solution constructed in Theorem \ref{thm-3.1}. One has 
\begin{equation*}
u\not \equiv 0.
\end{equation*}
\end{proposition}
\begin{proof} Suppose that $u \equiv 0$. Due to the second equation of \eqref{1.1} one has $v(x_1,0)=0$ and from the first equation of \eqref{1.1} we get 
\begin{equation}\label{3.4}
\int_{\Om_\ell} D\nabla v\cdot \nabla \varphi ~dx  =  \int_{\Om_\ell} f(v)\varphi ~dx  ~~\forall \varphi \in V. 
\end{equation}
Consider then a small ball $B=B_{x_0}$ centered at $x_0\in \Gamma_0$. Set 
 \begin{equation*}
{\tilde v=  }
\begin{cases} v \text{ in } \Om_\ell \cap B, \cr
0 \text{ in the rest of the ball.}
\end{cases}
\end{equation*}
Let $\varphi \in { \cal D}(B)$. One has by \eqref{3.4}, 
\begin{equation*}
\int_{B} D\nabla {\tilde v}\cdot \nabla \varphi ~dx  =  \int_{\Om_\ell \cap B} D\nabla v\cdot \nabla \varphi ~dx= \int_{\Om_\ell\cap B} f(v)\varphi ~dx= \int_{ B} f({\tilde v})\varphi ~dx  ~~\forall \varphi \in  { \cal D}(B). 
\end{equation*}
Thus  
\begin{equation*}
-D\Delta {\tilde v} = f({\tilde v}) \text{ in } B. 
\end{equation*}
It is clear that ${\tilde v}$ and thus $f({\tilde v})$ are bounded and one has $f({\tilde v})\in L^\infty(B) \subset L^p(B)~ \forall p$. From the usual regularity theory it follows that 
$ {\tilde v} \in W^{2,p}(B) \subset C^{1,\alpha}(B)$. Since $f( {\tilde v}) \geq 0, f( {\tilde v}) \not \equiv 0$ it follows that ${\tilde v} > 0$ in $B$ (see \cite{GT}). Hence a contradiction. This shows the impossibility 
for $u$ to be identical to $0$ and this completes the proof of the proposition.
\end{proof}
\begin{remark} One can easily show (see \cite{DL}, \cite{C3}) that
\begin{equation*}
\lambda_1 = \lambda_1({\Om_\ell}) = \Big(\frac{\pi}{2\ell}\Big)^2 +  \Big(\frac{\pi}{L}\Big)^2.
\end{equation*}
Thus for a smooth function $f$ it is clear that \eqref{2.11} is satisfied if 
\begin{equation*}
\lambda_1 = \lambda_1({\Om_\ell}) <\frac{f'(0)}{D},
\end{equation*}
i.e. for $\ell$ and $L$ large enough.
\vskip .3 cm
Note that \eqref{2.13} (see also \eqref{1.5}) is satisfied in the case of the Fisher equation i.e. for 
\begin{equation*}
f(v) = v(1-v)
\end{equation*}
the Lipschitz character of $f$ being used only on a finite interval.
\end{remark}

\section{Some extension} 
We would like to show that our results extend in the case of a so called two road problem. More precisely  set 
\begin{equation*}
\Gamma'_0 = (-\ell, \ell)\times \{L\}, ~~\Gamma_1 = \partial \Om_\ell \backslash \{\Gamma_0 \cup \Gamma'_0\},
\end{equation*}
\begin{equation*}
V = \{v\in H^1( \Om_\ell) ~|~ v=0 \text{ on } \Gamma_1 \},
\end{equation*}
\begin{figure}[H]
	\centering
	\includegraphics[scale=1]{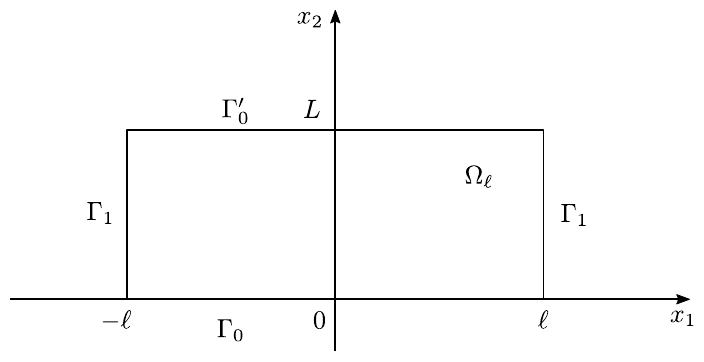}
	\caption{The domain $\Omega_\ell$ for two-road problem}
\end{figure}

We consider the problem of finding a $(u,v,w)$ solution to 
\begin{equation}\label{4.1}
\begin{cases}
(u,v,w) \in H^1_0(\Gamma_0)\times V \times H^1_0(\Gamma'_0), \cr

\int_{\Om_\ell} D\nabla v\cdot \nabla \varphi ~dx +  \int_{\Gamma_0} \nu v(x_1,0)\varphi  ~ dx_1 +\int_{\Gamma'_0} \nu v(x_1,L)\varphi  ~ dx_1\cr 
~~~~~~~~~~~~~~~~~~~~~~~~~~~~~~~~~~=  \int_{\Om_\ell} f(v)\varphi ~dx + \int_{\Gamma_0} \mu u \varphi ~dx_1 + \int_{\Gamma'_0} \mu' w \varphi ~dx_1~~\forall \varphi \in V, \cr

\int_{\Gamma_0} D'u'\psi'  +\mu u \psi~dx_1 = \int_{\Gamma_0} \nu v(x_1,0)\psi ~dx_1 +  \int_{\Gamma_0} g(u) \psi ~dx_1 ~~\forall \psi  \in H^1_0(\Gamma_0),\cr
\int_{\Gamma'_0} D''w'\phi'  +\mu' w \phi~dx_1 = \int_{\Gamma'_0} \nu' v(x_1,L)\phi~dx_1 +  \int_{\Gamma'_0} h(w) \phi ~dx_1 ~~\forall \phi  \in H^1_0(\Gamma'_0).
\end{cases} 
\end{equation}

Here we assume that 
\begin{equation*} D, D', D'', \mu,\nu, \mu',\nu' ~~\text{are positive constants,}
\end{equation*}
$f,g,h$ are Lipschitz continuous functions with Lipschitz constants $L_f,L_g,L_h$ respectively (Cf. \eqref{1.3}), which implies that for $\lambda \geq L_f$,  $\eta \geq L_g$ and $\xi\ge L_h$ the functions 
\begin{equation*} 
x \to \lambda x - f(x),~ \eta x - g(x), ~\xi x-h(x)
\end{equation*}
are nondecreasing.  We will suppose that $f$ satisfies \eqref{1.5} and that  
\begin{equation*}\label{4.2}
g(0)=0, ~~h(0)=0.
\end{equation*}
Since ${\Gamma_0}$ and ${\Gamma'_0}$ are playing exactly identical roles there is no loss of generality in assuming for instance
\begin{equation*}\label{4.3}
\frac{\mu}{\nu}  \geq \frac{\mu'}{\nu'}.
\end{equation*}
Then for 
\begin{equation*}\label{4.4}
m \geq  \frac{\nu}{\mu} , ~~m'=  \frac{\nu'}{\mu'} \frac{\mu}{\nu} m,
\end{equation*}
 we will assume 
\begin{equation}\label{4.7}
g(m)\leq 0, ~~h(m')\leq 0.
\end{equation}
One should notice the following properties 
\begin{equation*}\label{4.5}
 \frac{\mu}{\nu} m  \geq 1,~~  \frac{\mu'}{\nu'} m'  =\frac{\mu}{\nu} m.
\end{equation*}
\begin{equation*}\label{4.6}
m' =  \frac{\nu'}{\mu'} \frac{\mu}{\nu} m \geq  \frac{\nu'}{\mu'}.
\end{equation*}
Then with small variants we can reproduce the results we had in the preceding sections. First we have

\begin{lemma} Suppose that $\widetilde{u}, \widetilde w $ are measurable functions on ${\Gamma_0}$ and $\Gamma'_0$ respectively such that 
\begin{equation*}\label{4.8}
0\leq \widetilde{u} \leq m,~~ 0\leq \widetilde{w} \leq m'. \end{equation*}
Then under the assumptions above  the problem 
\begin{equation}\label{4.9}
\begin{cases}
v\in V, \cr
\int_{\Om_\ell} D\nabla v\cdot \nabla \varphi ~dx +\int_{\Gamma_0} \nu  v(x_1,0)\varphi (x_1,0) ~ dx_1 + \int_{\Gamma'_0} \nu'  v(x_1,L)\varphi (x_1,L) ~ dx_1\cr 
\qquad\qquad\qquad=  \int_{\Om_\ell} f(v)\varphi ~dx + \int_{\Gamma_0} \mu \widetilde u \varphi(x_1,0) ~dx_1+ \int_{\Gamma'_0} \mu' \widetilde w \varphi(x_1,L) ~dx_1 ~~\forall \varphi \in V, \cr
\end{cases} 
\end{equation}
possesses a minimal and a maximal solution with values in $(0,\frac{\mu}{\nu}m)$.
\end{lemma}
\begin{proof} Let us remark first that any nonnegative solution to \eqref{4.9} takes its values in $(0, \frac{\mu}{\nu}m)$. Indeed if $v$ is solution to \eqref{4.9} taking as test function $\varphi = (v-k)^+$, $k=\frac{\mu}{\nu}m \geq 1$ one gets 
\begin{equation*}
\begin{aligned}
&\int_{\Om_\ell} D|\nabla (v-k)^+|^2 ~dx = \int_{\Om_\ell} D\nabla (v-k)\cdot \nabla (v-k)^+~dx = \int_{\Om_\ell} D\nabla v\cdot \nabla (v-k)^+~dx =\cr
&   \int_{\Om_\ell} f(v)(v-k)^+~dx + \int_{\Gamma_0} \{\mu \widetilde u - \nu v(x_1,0)\}(v-k)^+ ~ dx_1+ \int_{\Gamma'_0} \{\mu' \widetilde w - \nu' v(x_1,L)\}(v-k)^+ ~ dx_1 \cr
 & \leq  \int_{\Gamma_0} \{\mu \widetilde u - \nu v(x_1,0) \}(v-k)^+ ~ dx_1+\int_{\Gamma'_0} \{\mu' \widetilde w - \nu' v(x_1,L)\}(v-k)^+ ~ dx_1 \leq 0, \cr
\end{aligned}
\end{equation*}
since on the set where $v \geq k = \frac{\mu}{\nu}m = \frac{\mu'}{\nu'}m'$ one has  $\{\mu \widetilde u - \nu v(x_1,0)\}\leq \{\mu \widetilde u - \mu m\} \leq0$ and 
$\{\mu'\widetilde w - \nu' v(x_1,L)\}\leq \{\mu' \widetilde w - \mu' m'\} \leq0$. 

\noindent Next let us note that $0$ is a subsolution to \eqref{4.9}. Indeed $\forall \varphi \in V,  \varphi  \geq 0$, one has
\begin{equation*}\label{4.10}
\begin{aligned}
\int_{\Om_\ell} D\nabla 0\cdot \nabla \varphi ~dx +  \int_{\Gamma_0} \nu 0\varphi ~ dx_1&+  \int_{\Gamma'_0} \nu' 0\varphi ~ dx_1 \cr 
&\leq  \int_{\Om_\ell} f(0)\varphi ~dx + \int_{\Gamma_0} \mu \widetilde u \varphi ~dx_1 + \int_{\Gamma'_0} \mu' \widetilde w \varphi ~dx_1.
\end{aligned}
\end{equation*}
On the other hand,  $k= \frac{\mu}{\nu}m$ is a supersolution since for $\varphi \in V$, $\varphi \geq 0$, 
\begin{equation*}
\int_{\Om_\ell} D \nabla k\cdot \nabla \varphi ~dx + \int_{\Gamma_0} \nu k \varphi ~dx_1+ \int_{\Gamma'_0} \nu' k \varphi ~dx_1 \geq \int_{\Om_\ell} f( k) \varphi ~dx + \int_{\Gamma_0}\mu \widetilde u \varphi ~dx_1 + \int_{\Gamma'_0}\mu' \widetilde w \varphi ~dx_1.
\end{equation*}
\vskip .3 cm
For $z\in L^2(\Om_\ell)$ we denote by $y=S(z)$ the solution to 
\begin{equation*}\label{4.11}
\begin{cases}
y\in V, \cr
\int_{\Om_\ell} D\nabla y\cdot \nabla \varphi ~dx + \int_{\Om_\ell} \lambda y \varphi ~dx + \int_{\Gamma_0} \nu y(x_1,0)\varphi  ~ dx_1 +\int_{\Gamma'_0} \nu' y(x_1,L)\varphi  ~ dx_1  \cr \qquad\qquad\qquad=  \int_{\Om_\ell} f(z)\varphi ~dx + \int_{\Om_\ell} \lambda z \varphi ~dx  +\int_{\Gamma_0} \mu \widetilde u \varphi ~dx_1 +\int_{\Gamma'_0} \mu' \widetilde w \varphi ~dx_1 ~~\forall \varphi \in V, \cr
\end{cases} 
\end{equation*}
where $\lambda \geq L_f$. The existence of a unique solution $y$ to the problem above follows from the Lax-Milgram theorem. 
Then reproducing the arguments of Lemma \ref{lemma-2.1} it is easy to show that $S$ is continuous and monotone. Introducing the sequence defined in  \eqref{2.7} one concludes as in the Lemma \ref{lemma-2.1} to the existence 
of a minimal and maximal solution $ \underline{v}$ and $ \overline{v}$.
\end{proof}
Then one has :
\begin{lemma} Under the assumptions of the preceding lemma and \eqref{2.11}, for $\epsilon>0$ small enough, the maximal solution $\overline v$ to \eqref{4.9} satisfies 
\begin{equation*}\label{4.12}
\epsilon \varphi_1  \leq \overline v.
\end{equation*}
 In particular $\overline v$ is bounded away from $0$.
\end{lemma}

\begin{proof}
Due to \eqref{2.11} one has for $\epsilon>0$ small enough
\begin{equation*}
D\lambda_1 \epsilon \varphi_1 \leq f( \epsilon \varphi_1).
\end{equation*}
Then  for $\varphi \in V$, $\varphi \geq 0$ it holds after integration by parts
\begin{equation*}
\begin{aligned}
\int_{\Om_\ell} D \nabla( \epsilon \varphi_1)\cdot \nabla \varphi ~dx &+ \int_{\Gamma_0} \nu \epsilon \varphi_1(x_1,0) \varphi ~dx_1 + \int_{\Gamma'_0} \nu' \epsilon \varphi_1(x_1,L) \varphi ~dx_1 \cr
&=   \int_{\Om_\ell} D \nabla( \epsilon \varphi_1)\cdot \nabla \varphi ~dx =  \int_{\Om_\ell}  \nabla \cdot (D \nabla (\epsilon \varphi_1) \varphi )- D\Delta (\epsilon \varphi_1) \varphi  ~dx\cr
 &  = \int_{\partial \Om_\ell}D\partial_{n} (\epsilon \varphi_1)  \varphi d\sigma +  \int_{\Om_\ell}  D \lambda_1 (\epsilon \varphi_1) \varphi  ~dx\cr
 &\leq  \int_{\Om_\ell}  f(\epsilon \varphi_1) \varphi  ~dx + \int_{\Gamma_0} \mu {\tilde u} \varphi ~dx_1 +  \int_{\Gamma'_0} \mu' {\tilde w} \varphi ~dx_1. \cr
\end{aligned}
\end{equation*}
($n$ denotes the outward unit normal to $\Om_\ell$, note that $\partial_{n} (\epsilon \varphi_1) \leq 0$). Thus $\epsilon \varphi_1$ is a positive subsolution to \eqref{4.9} and one concludes as in the proof of Lemma \ref{lemma-2.2}.
\end{proof}

Analogous to Lemma \ref{lemma-2.3} one has : 

\begin{lemma}\label{lemma-4.3} Suppose that $f$ satisfies \eqref{2.13}.
If $v_1, v_2$ are positive solutions to \eqref{4.9} corresponding to $(u_1,w_1)$ and $(u_2,w_2)$ respectively then 
\begin{equation*}\label{4.13}
u_1\leq u_2 \text{ and } w_1 \leq w_2 \text{ implies }  v_1 \leq v_2.
\end{equation*}
In particular \eqref{4.9} has a unique positive solution. 
\end{lemma}
\begin{proof} Denote by $\theta$ a smooth function such that 
\begin{equation*}
\theta(t) =0 ~~\forall  t\leq 0,~~ \theta(t) =1 ~~ \forall  t\geq 1, ~~\theta'(t) \geq 0.
\end{equation*}
Set $\theta_\epsilon(t)= \theta(\frac{t}{\epsilon})$. Clearly 
\begin{equation*}
v_1\theta_\epsilon(v_1-v_2),~v_2\theta_\epsilon(v_1-v_2) \in V.
\end{equation*}

From the equations satisfied by $v_1,v_2$ one gets, setting $\theta_\epsilon =\theta_\epsilon(v_1-v_2) $,
\begin{align*}
&\int_{\Om_\ell} D\nabla v_1\cdot \nabla (v_2\theta_\epsilon) ~dx +  \int_{\Gamma_0} \nu v_1(x_1,0)(v_2\theta_\epsilon) ~ dx_1+\int_{\Gamma'_0} \nu' v_1(x_1,L)(v_2\theta_\epsilon) ~ dx_1\cr
&\qquad\qquad\qquad\qquad\qquad =  \int_{\Om_\ell} f(v_1)(v_2\theta_\epsilon) ~dx + \int_{\Gamma_0} \mu u_1(v_2\theta_\epsilon) ~dx_1+ \int_{\Gamma'_0} \mu' w_1(v_2\theta_\epsilon) ~dx_1,
\end{align*}
\begin{align*}
&\int_{\Om_\ell} D\nabla v_2\cdot \nabla (v_1\theta_\epsilon) ~dx +  \int_{\Gamma_0} \nu v_2(x_1,0)(v_1\theta_\epsilon) ~ dx_1+\int_{\Gamma'_0} \nu' v_2(x_1,L)(v_1\theta_\epsilon) ~ dx_1 \cr
&\qquad\qquad\qquad\qquad\qquad =  \int_{\Om_\ell} f(v_2)(v_1\theta_\epsilon) ~dx + \int_{\Gamma_0} \mu u_2(v_1\theta_\epsilon) ~dx_1+ \int_{\Gamma'_0} \mu' w_2(v_1\theta_\epsilon) ~dx_1.
\end{align*}
By subtraction we obtain
\begin{equation*}
\begin{aligned}
&\int_{\Om_\ell} D\{\nabla v_2\cdot \nabla (v_1\theta_\epsilon) -\nabla v_1\cdot \nabla (v_2\theta_\epsilon)\}~dx  =  \int_{\Om_\ell} f(v_2)(v_1\theta_\epsilon)- f(v_1)(v_2\theta_\epsilon) ~dx\cr 
&\qquad\qquad\qquad\qquad + \int_{\Gamma_0} \mu (u_2v_1-u_1v_2)\theta_\epsilon (v_1-v_2)~dx_1+ \int_{\Gamma'_0} \mu' (w_2v_1-w_1v_2)\theta_\epsilon (v_1-v_2)~dx_1.
\end{aligned}
\end{equation*}
Clearly the last two integrals above are nonnegative so that one has 
\begin{equation*}
\int_{\Om_\ell} D\{\nabla v_2\cdot \nabla (v_1\theta_\epsilon) -\nabla v_1\cdot \nabla (v_2\theta_\epsilon)\}~dx  \geq  \int_{\Om_\ell} f(v_2)(v_1\theta_\epsilon)- f(v_1)(v_2\theta_\epsilon) ~dx .
 \end{equation*}
Then the rest of the proof is like in Lemma \ref{lemma-2.3}.
\end{proof}

Then we can now show :
\begin{theorem}
Under the assumptions above the problem \eqref{4.1} admits a nontrivial solution.
\end{theorem}
\begin{proof} It is of course clear that $(0,0,0)$ is solution to \eqref{4.1}. Set 
 \begin{equation*}
K = \{u \in L^2({\Gamma_0}) ~|~ 0\leq u\leq m\},~~K' = \{w \in L^2({\Gamma'_0}) ~|~ 0\leq w \leq m'\}.
 \end{equation*}
 For $(u,w) \in K\times K'$, let $\overline v$ be the unique positive solution to \eqref{4.9} associated to $(\widetilde u,\widetilde w)=(u,w)$. For $\eta \geq L_g$, $\xi\geq L_h$, let $(U,W)=T(u,w)$ be the solution to 
 \begin{equation*}\label{4.14}
\begin{cases}
(U,W)\in H^1_0(\Gamma_0)\times H^1_0(\Gamma'_0), \cr
\int_{\Gamma_0} D'U'\psi'  +\mu U \psi + \eta U \psi~dx_1 = \int_{\Gamma_0} \nu \overline{v}(x_1,0)\psi +  g(u) \psi + \eta u \psi~dx_1  ~\forall \psi  \in H^1_0(\Gamma_0), \cr
\int_{\Gamma'_0} D''W'\phi'  +\mu' W \phi + \xi W \phi~dx_1 \cr
~~~~~~~~~~~~~~~~~~~~~~~~~~~~~~~~~~~~~~~= \int_{\Gamma'_0} \nu' \overline{v}(x_1,L)\phi +  h(w) \phi + \xi w \phi~dx_1  ~\forall \phi  \in H^1_0(\Gamma'_0). \cr
\end{cases} 
\end{equation*}
The existence of $(U,W)$ is a consequence of the Lax-Milgram theorem. 
\vskip .3 cm
We show as in Theorem \ref{thm-3.1}  that $T$ is continuous on $K\times K'\subset L^2({\Gamma_0})\times L^2({\Gamma'_0})$. Indeed suppose that $u_n \to u$ in $K$ and $w_n \to w$ in $K'$. Denote by $\overline{v}_n$ the solution to \eqref{4.9} associated to $(u_n,w_n)$ i.e. satisfying 
 \begin{align}\label{4.15}
&\int_{\Om_\ell} D\nabla\overline{v}_n\cdot \nabla \varphi ~dx +  \int_{\Gamma_0} \nu\overline{v}_n(x_1,0)\varphi  ~ dx_1 +\int_{\Gamma'_0}\nu'\overline{v}_n(x_1,L)\varphi  ~ dx_1\cr
&\qquad\qquad\qquad\qquad=  \int_{\Om_\ell} f(\overline{v}_n)\varphi ~dx + \int_{\Gamma_0} \mu u_n \varphi ~dx_1 + \int_{\Gamma'_0} \mu' w_n \varphi ~dx_1 ~~\forall \varphi \in V.
\end{align}
Since $\overline{v}_n$, $u_n$ and $w_n$ are bounded, taking $\varphi= \overline{v}_n$ in the equality above one gets easily 
\begin{equation*}
\int_{\Om_\ell} D|\nabla\overline{v}_n|^2~dx +  \int_{\Gamma_0} \nu\overline{v}_n(x_1,0)^2  ~ dx_1+\int_{\Gamma'_0}\nu'\overline{v}_n(x_1,L)^2  ~ dx_1 \leq C,
\end{equation*}
where $C$ is a constant independent of $n$. Thus, up to a subsequence, there exists $v \in V$ such that 
\begin{equation*}\label{4.16}
\begin{aligned}
\overline{v}_n \rightharpoonup v \text{ in } H^1({\Om_\ell}), &~~\overline{v}_n \to v \text{ in } L^2({\Om_\ell}),\cr
&~~\overline{v}_n(.\,,0)  \to v(.\,,0) \text{ in } L^2({\Gamma_0}),~~\overline{v}_n(.\,,L)  \to v(.\,,L) \text{ in } L^2({\Gamma'_0}).
\end{aligned}
\end{equation*}
Passing to the limit in \eqref{4.15} one derives as in Theorem \ref{thm-3.1} that $T(u_n,w_n) \to T(u,w)$ in $ L^2({\Gamma_0})\times L^2({\Gamma'_0})$. 
\vskip .3 cm
We can show also that $T$ is monotone. Indeed, suppose that $(u_1,w_1) \geq (u_2, w_2)$ in the sense that $u_1 \geq u_2 $ and $w_1 \geq w_2$ and set $(U_i,W_i)=T(u_i,w_i)$, $i=1,2$. First, for $U_i$ one has 
  \begin{equation*}
  \begin{aligned}
\int_{\Gamma_0} D'U_1'\psi'  +\mu U_1 \psi + \eta U_1 \psi~dx_1 = \int_{\Gamma_0} \nu \overline{v}_1(x_1,0)\psi +  g(u_1) \psi + \eta u_1 \psi~dx_1  ~~\forall \psi  \in H^1_0(\Gamma_0), \cr
\int_{\Gamma_0} D'U_2'\psi'  +\mu U_2 \psi + \eta U_2 \psi~dx_1 = \int_{\Gamma_0} \nu \overline{v}_2(x_1,0)\psi +  g(u_2) \psi + \eta u_2 \psi~dx_1  ~~\forall \psi  \in H^1_0(\Gamma_0). \cr
 \end{aligned}
\end{equation*}
By subtraction it comes 
  \begin{equation*}
  \begin{aligned}
\int_{\Gamma_0}& D'(U_1-U_2)'\psi'  +\mu (U_1-U_2) \psi + \eta (U_1-U_2) \psi~dx_1 \cr &= \int_{\Gamma_0} \nu (\overline{v}_1(x_1,0)-\overline{v}_2(x_1,0))\psi +  (g(u_1)-g(u_2)) \psi + \eta (u_1-u_2) \psi~dx_1  ~~\forall \psi  \in H^1_0(\Gamma_0). \cr
\end{aligned}
\end{equation*}
Choosing $\psi= -(U_1-U_2)^-$ and taking into account that, by Lemma \ref{lemma-4.3}, $\overline{v}_1(x_1,0)-\overline{v}_2(x_1,0)\geq 0$ and that for $\eta \geq L_g$, $(g(u_1)-g(u_2)) + \eta (u_1-u_2) \geq 0$ (Cf. \eqref{1.3},
\eqref{1.4}), one gets
 \begin{equation*}
\int_{\Gamma_0} D'|\{(U_1-U_2)^-\}'|^2  +\mu \{(U_1-U_2)^-\}^2  + \eta \{(U_1-U_2)^-\}^2~dx_1 \leq 0.
\end{equation*}
Thus $(U_1-U_2)^-=0$ and $U_1 \geq U_2$. Similarly one shows that   $W_1 \geq W_2$.
\vskip .3 cm
Next we assert that $T$ maps $K\times K'$ into itself. Indeed, if $(U_0,W_0)=T(0,0)$ one has, with an obvious notation for $\overline{v}_0$
\begin{equation*}
\int_{\Gamma_0} D'U_0'\psi'  +\mu U_0 \psi + \eta U_0 \psi~dx_1 = \int_{\Gamma_0} \nu \overline{v}_0(x_1,0)\psi ~dx_1  ~~\forall \psi  \in H^1_0(\Gamma_0).
\end{equation*}
\begin{equation*}
\int_{\Gamma'_0} D''W_0'\psi'  +\mu W_0 \phi + \xi W_0 \phi~dx_1 = \int_{\Gamma'_0} \nu' \overline{v}_0(x_1,L)\phi ~dx_1  ~~\forall \phi  \in H^1_0(\Gamma'_0).
\end{equation*}
Taking $\psi = -U_0^-$, $\phi=-W_0^-$, one deduces easily since $\overline{v}_0(x_1,0) \geq 0$ and $\overline{v}_0(x_1,L) \geq 0$ that $U_0 \geq 0$ and $W_0\geq 0$.
Similarly, if $(U_1,W_1) = T(m,m')$ one has, with an obvious notation for $\overline{v}_1$
\begin{equation*}
\int_{\Gamma_0} D'U_1'\psi'  +\mu U_1 \psi + \eta U_1 \psi~dx_1 = \int_{\Gamma_0} \nu  \overline{v}_1(x_1,0)\psi +g(m)\psi+ \eta m \psi~dx_1  ~~\forall \psi  \in H^1_0(\Gamma_0),
\end{equation*}
\begin{equation*}
\int_{\Gamma'_0} D''W_1'\phi'  +\mu' W_1 \phi + \xi W_1 \phi~dx_1 = \int_{\Gamma'_0} \nu'  \overline{v}_1(x_1,L)\phi +h(m')\phi+ \xi m' \phi~dx_1  ~~\forall \phi  \in H^1_0(\Gamma'_0).
\end{equation*}
Thus choosing $\psi = (U_1-m)^+$ and $\phi=(W_1-m')^+$, due to \eqref{4.7} and $\overline{v}_1 \leq \frac{\mu}{\nu}m = \frac{\mu'}{\nu'}m'$ it comes 
\begin{equation*}
 \begin{aligned}
\int_{\Gamma_0} D'|\{(U_1-m)^+\}'|^2  &+(\mu  + \eta) \{(U_1-m)^+\}^2   \cr&  \leq \int_{\Gamma_0} (\nu  \overline{v}_1(x_1,0) -\mu m) (U_1-m)^+dx_1 \leq 0.
 \end{aligned}
\end{equation*}
\begin{equation*}
\begin{aligned}
\int_{\Gamma'_0} D''|\{(W_1-m')^+\}'|^2  &+(\mu'  + \xi) \{(W_1-m')^+\}^2   \cr
& \leq \int_{\Gamma'_0} (\nu'  \overline{v}_1(x_1,L) -\mu' m') (W_1-m')^+dx_1 \leq 0.
\end{aligned}
\end{equation*}
From which it follows that $U_1 \leq m$, $W_1 \leq m'$. By the monotonicity of $T$ it results that $T$ maps the convex $K\times K'$ into itself. But clearly $T(K\times K') \subset C^\frac{1}{2}(\Gamma_0)\times C^\frac{1}{2}(\Gamma'_0)$  is relatively compact in 
$L^2(\Gamma_0)\times L^2(\Gamma'_0)$. Thus, by the Schauder fixed point theorem (see \cite{E},  \cite{GT},  \cite{C}), $T$ has a fixed point in $K\times K'$ which leads to a nontrivial solution $(u,v,w)$ to \eqref{4.1}. This completes the proof of the theorem.
\end{proof}
\begin{remark}
One can show as in Proposition \ref{prop-3.1} that $u$, $w$ are also non degenerate in the sense that 
\begin{equation*}
u \not \equiv 0,~~ w \not \equiv 0.
\end{equation*}
\end{remark}

\section{The case of an unbounded domain} 

The goal of this section is to show that when $\ell=+\infty$ it remains possible to define and find a nontrivial solution to problem \eqref{1.1}. Let us introduce some notation. For convenience we will 
denote by $V_\ell$ the space $V$ defined in section 1. Similarly we will indicate the dependence in $\ell$ for $\Gamma_0$ i.e 
\begin{equation*}
\Gamma_0 =\Gamma^\ell_0 = (-\ell, \ell)\times \{0\}.
\end{equation*}
When convenient we will set $I_\ell= (-\ell,\ell)$. In addition we set 
\begin{equation*}
\Om_\infty = \R \times (0,L), ~~\Gamma^\infty_0 = \R\times \{0\}, ~~\Gamma^\infty_1 = \R\times \{L\},
\end{equation*}
\begin{equation*}
V_\infty =\{ v \in H^1_{\ell oc}(\overline{\Om_\infty})~|~ v=0 \text{ on } \Gamma^\infty_1\},
\end{equation*}
where
\begin{equation*}
H^1_{\ell oc}(\overline{\Om_\infty})=\{v~|~ v\in H^1(\Om_{\ell_0}) ~\forall \ell_0 >0\}.
\end{equation*}
Then we have 
\begin{theorem} Under the assumption above, in particular  \eqref{1.2}, \eqref{1.3}, \eqref{1.5}, \eqref{1.6}, \eqref{2.11}, \eqref{2.13}  there exists $(u,v)$ nontrivial solution to 
\begin{equation}\label{5.1}
\begin{cases}
(u,v) \in H^1_0(\Gamma^\infty_0)\times V_\infty, \cr

\int_{\Om_{\ell_0}} D\nabla v\cdot \nabla \varphi ~dx +  \int_{I_{\ell_0}} \nu v(x_1,0)\varphi  ~ dx_1 \cr 

~~~~~~~~~~~~~~~~~~~~~~~~~~~~~~~~~~~~~~~~=  \int_{\Om_{\ell_0}} f(v)\varphi ~dx + \int_{I_{\ell_0}}  \mu u \varphi ~dx_1 ~~\forall \varphi \in V_{\ell_0}, ~\forall {\ell_0}, \cr

\int_{I_{\ell_0}}  D'u'\psi'  +\mu u \psi~dx_1 = \int_{I_{\ell_0}}  \nu v(x_1,0)\psi ~dx_1 +  \int_{I_{\ell_0}}  g(u) \psi ~dx_1 ~~\forall \psi  \in H^1_0({I_{\ell_0}} ),~\forall {\ell_0}. \cr
\end{cases} 
\end{equation}
(We  identify  $\Gamma^\infty_0$ with $\R$. Recall that $I_\ell= (-\ell,\ell)$).
\end{theorem}
\begin{proof} Let $(u_\ell, v_\ell)$ be a solution to \eqref{1.1}. We can find such a solution for every $\ell >0$ (Cf. Theorem \ref{thm-3.1}). One notices that for $\ell' \geq \ell$ one has 
\begin{equation*}
\Om_\ell \subset \Om_{\ell'}, ~~H^1_{0}({\Om_\ell}) \subset H^1_{0}({\Om_{\ell'}}),
\end{equation*}
(we suppose the functions of $H^1_{0}({\Om_\ell})$ extended by $0$ outside ${\Om_\ell}$). By definition of $\lambda_1 = \lambda_1(\Om_\ell)$ one has 
\begin{equation*}
\lambda_1(\Om_\ell) =  \inf_{H^1_{0}({\Om_\ell})\backslash \{0\}} \frac{\int_{\Om_\ell} |\nabla v |^2 dx}{\int_{\Om_\ell} v^2dx} ,
\end{equation*}
and thus clearly 
\begin{equation*}
\lambda_1(\Om_{\ell'}) \leq \lambda_1(\Om_\ell) ~~\forall~ \ell' \geq \ell.
\end{equation*}
Let us assume for some $\ell_1 >0$ (Cf. \eqref{2.11}) 
\begin{equation}\label{5.2}
 \lambda_1(\Om_{\ell_1}) \leq \frac{f(s)}{Ds} ~~\text{ for } s >0 \text{ small enough. }
\end{equation}
Then for any $\ell \geq \ell_1$ one has for for $s >0$ small enough
\begin{equation*}
 \lambda_1(\Om_{\ell}) \leq \frac{f(s)}{Ds} .
\end{equation*}
Moreover, since it is easy to see that $\varphi_1$ defined in \eqref{2.10} is given by 
\begin{equation*}
\varphi_1 =\sin\frac{\pi}{2\ell}(x_1+\ell)\sin\frac{\pi}{L}x_2,
\end{equation*}
one has $ 0\leq \varphi_1  \leq 1$ and if \eqref{5.2} holds one has
\begin{equation*}
D \lambda_1(\Om_{\ell})\epsilon \varphi_1 \leq f(\epsilon \varphi_1) 
\end{equation*}
for $\epsilon>0$ small enough independently of $\ell \geq \ell_1$. We suppose from now on that this $\epsilon$ is fixed such that if $(u_\ell,v_\ell)$ is a solution to \eqref{1.1} constructed as in Theorem \ref{thm-3.1} one has 
\begin{equation*}
\epsilon \varphi_1 \leq v_\ell 
\end{equation*}
and in particular for every $\ell \geq \ell_1$
\begin{equation}\label{5.3}
\epsilon(\sin\frac{\pi}{4})^2 \leq \epsilon \varphi_1 \leq v_\ell~~ \text{ a.e. } x \in (-\frac{\ell}{2},\frac{\ell}{2})\times (\frac{L}{4},\frac{3L}{4}).
\end{equation}
One should also notice that independently of $\ell$ one has 
\begin{equation}\label{5.4}
0\leq u_\ell \leq m,~~  \epsilon \varphi_1 \leq v_\ell\leq \frac{\mu}{\nu}m.
\end{equation}
We assume from now on $\ell \geq \ell_1$ and for $  \ell_0 \leq \ell - 1$ we define $\rho$ by 
\begin{equation*}
\rho = \rho(x_1)=
\begin{cases}
1 \text{ on } I_{\ell_0},\cr
x_1+\ell_0 +1 \text{ on } (-\ell_0-1, -\ell_0),\cr
-x_1 + \ell_0 +1 \text{ on } (\ell_0,\ell_0 +1),\cr
0 \text{ outside }I_{\ell_0+1},\cr
\end{cases}
\end{equation*}
whose  graph is depicted below.
\begin{figure}[H]
	\centering
	\includegraphics[scale=1.5]{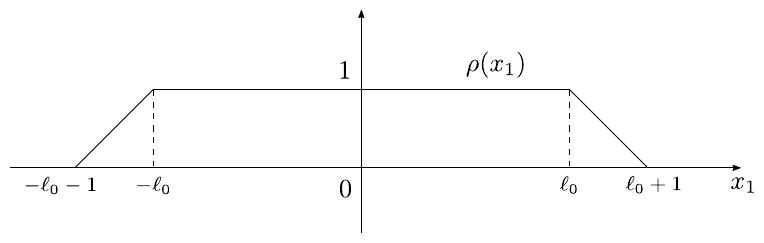}
	\caption{The graph of the function  $\rho(x_1)$}
\end{figure}
Clearly $\rho^2 v_\ell = \rho^2(x_1) v_\ell \in V_\ell$ and from the first equation of \eqref{1.1} one gets
\begin{equation}\label{5.5}
\int_{\Om_\ell} D\nabla v_\ell \cdot \nabla (\rho^2 v_\ell) ~dx +  \int_{\Gamma_0} \nu \rho^2 v_\ell^2(x_1,0) ~ dx_1 =  \int_{\Om_\ell} f(v_\ell)\rho^2 v_\ell~dx + \int_{\Gamma_0} \mu u_\ell \rho^2 v_\ell (x_1,0)~dx_1.
\end{equation}
One should notice that in the integrals in $\Om_\ell$ one integrates only on $\Om_{\ell_0+1}$ and for the ones in $\Gamma_0$ on $I_{\ell_0+1}$. Then remark that 
\begin{equation*}
\int_{\Om_\ell} \nabla v_\ell \cdot \nabla (\rho^2 v_\ell) ~dx = \int_{\Om_\ell} |\nabla v_\ell|^2\rho^2 + 2\rho v_\ell \nabla v_\ell \cdot \nabla \rho~dx, 
\end{equation*}
and 
\begin{equation*}
\int_{\Om_\ell}| \nabla (\rho v_\ell)|^2 ~dx = \int_{\Om_\ell}|\rho \nabla  v_\ell +  v_\ell  \nabla   \rho |^2 ~dx = \int_{\Om_\ell} |\nabla v_\ell|^2\rho^2 + 2\rho v_\ell \nabla v_\ell \cdot \nabla \rho + v_\ell^2 |\nabla \rho|^2~dx. 
\end{equation*}
From this it follows that 
\begin{equation*}
\int_{\Om_\ell} \nabla v_\ell \cdot \nabla (\rho^2 v_\ell) ~dx = \int_{\Om_\ell}| \nabla (\rho v_\ell)|^2 ~dx - \int_{\Om_\ell} v_\ell^2 |\nabla \rho|^2~dx. 
\end{equation*}
Thus, since the second integral of \eqref{5.5} is nonnegative,  it comes 
\begin{equation*}
D\int_{\Om_{\ell_0+1}}| \nabla (\rho v_\ell)|^2 ~dx \leq D \int_{\Om_{\ell_0+1}} v_\ell^2 |\nabla \rho|^2 ~dx + \int_{\Om_{\ell_0+1}} f(v_\ell)\rho^2 v_\ell~dx 
+ \int_{I_{\ell_0+1}} \mu u_\ell \rho^2 v_\ell (x_1,0)~dx_1. 
\end{equation*}
Using the definition of $\rho$ and in particular the fact that $\rho =1 $ on $\Om_{\ell_0}$ we get easily by \eqref{5.4}
\begin{equation}\label{5.6}
\int_{\Om_{\ell_0}}| \nabla  v_\ell|^2 ~dx \leq C
\end{equation}
where $C$ is independent of $\ell$. 
Taking now $\psi = \rho^2u_\ell$ in the second equation of \eqref{1.1} we get
\begin{equation*}
\int_{I_{\ell_0+1}} D'u_\ell'( \rho^2u_\ell)'  +\mu \rho^2u_\ell^2~dx_1 = \int_{I_{\ell_0+1}} \nu v_\ell(x_1,0)\rho^2u_\ell +  g(u_\ell) \rho^2u_\ell ~dx_1.
\end{equation*}
Arguing as above we derive easily 
\begin{equation*}
\int_{I_{\ell_0+1}} u_\ell'  (\rho^2 u_\ell)' ~dx_1 = \int_{I_{\ell_0+1}}| (\rho u_\ell)'|^2 ~dx_1 - \int_{I_{\ell_0+1}} u_\ell^2 \rho'^2~dx_1. 
\end{equation*}
This leads to 
\begin{equation*}
 \int_{I_{\ell_0+1}}D' | (\rho u_\ell)'|^2 + \mu \rho^2u_\ell^2~dx_1 \leq\int_{I_{\ell_0+1}} D'u_\ell^2 \rho'^2~dx_1 + \nu v_\ell(x_1,0)\rho^2u_\ell +  g(u_\ell) \rho^2u_\ell ~dx_1. 
\end{equation*}
Integrating only on $I_{\ell_0}$ in the fisrt integral i.e. where $\rho=1$ we obtain
\begin{equation}\label{5.7}
 \int_{I_{\ell_0}}( u_\ell')^2 +  u_\ell^2~dx_1 \leq C 
\end{equation}
where $C$ is some other constant independent of $\ell$.
It results from \eqref{5.6}, \eqref{5.7} that $(u_\ell,v_\ell)$ is bounded in $H^1(I_{\ell_0})\times V_{\ell_0}$ independently of $\ell$. Thus there exists a subsequence of $(u_\ell,v_\ell)$ that we will 
denote by $(u_{n,0},v_{n,0})$ such when $n\to \infty$
\begin{equation*}
u_{n,0} \rightharpoonup u^0 \text{ in } H^1(I_{\ell_0}), ~~v_{n,0} \rightharpoonup v^0 \text{ in } V_{\ell_0}, ~~u_{n,0} \to u^0 \text{ in } L^2(I_{\ell_0}), ~~v_{n,0} \to v^0 \text{ in } L^2(\Om_{\ell_0})
\end{equation*}
\begin{equation*}
v_{n,0}(.,0) \to v^0(.,0) \text{ in } L^2(I_{\ell_0}).
\end{equation*}
Considering the equations 
\begin{equation*}
\int_{\Om_{\ell_0}} D\nabla v_\ell\cdot \nabla \varphi ~dx +  \int_{I_{\ell_0}} \nu v_\ell(x_1,0)\varphi  ~ dx_1 =  \int_{\Om_{\ell_0}} f(v_\ell)\varphi ~dx + \int_{I_{\ell_0}} \mu u_\ell \varphi ~dx_1 ~~\forall \varphi \in V_{\ell_0},
\end{equation*}
\begin{equation*}
\int_{I_{\ell_0}} D'u_\ell'\psi'  +\mu u_\ell \psi~dx_1 = \int_{I_{\ell_0}} \nu v_\ell(x_1,0)\psi ~dx_1 +  \int_{I_{\ell_0}} g(u_\ell) \psi ~dx_1 ~~\forall \psi  \in H^1_0(I_{\ell_0}),
\end{equation*}
with $(u_\ell, v_\ell)$ replaced by $(u_{n,0},v_{n,0})$, one can pass to the limit in $n$ and see that $(u^0,v^0) \in H^1(I_{\ell_0}) \times V_{\ell_0}$ satisfies 
\begin{equation*}
\int_{\Om_{\ell_0}} D\nabla v^0\cdot \nabla \varphi ~dx +  \int_{I_{\ell_0}} \nu v^0(x_1,0)\varphi  ~ dx_1 =  \int_{\Om_{\ell_0}} f(v^0)\varphi ~dx + \int_{I_{\ell_0}} \mu u^0 \varphi ~dx_1 ~~\forall \varphi \in V_{\ell_0},
\end{equation*}
\begin{equation*}
\int_{I_{\ell_0}} D'{u^0}'\psi'  +\mu u^0 \psi~dx_1 = \int_{I_{\ell_0}} \nu v^0(x_1,0)\psi ~dx_1 +  \int_{I_{\ell_0}} g(u^0) \psi ~dx_1 ~~\forall \psi  \in H^1_0(I_{\ell_0}).
\end{equation*}
(Note that a function of $V_{\ell_0}$ extended by $0$ belongs to $V_\ell$).
Clearly -as a subsequence of $(u_\ell,v_\ell)$- the sequence $(u_{n,0},v_{n,0})$ is bounded in $H^1(I_{\ell_0+1})\times V_{\ell_0+1}$ independently of $n$ and one can extract a subsequence that we still label by $n$ 
and denote by $(u_{n,1},v_{n,1})$ such that 
\begin{equation*}
u_{n,1} \rightharpoonup u^1 \text{ in } H^1(I_{\ell_0+1}), ~~v_{n,1} \rightharpoonup v^1 \text{ in } V_{\ell_0+1}, ~~u_{n,1} \to u^1 \text{ in } L^2(I_{\ell_0+1}), ~~v_{n,1} \to v^1 \text{ in } L^2(\Om_{\ell_0+1})
\end{equation*}
\begin{equation*}
v_{n,1}(.,0) \to v^1(.,0) \text{ in } L^2(I_{\ell_0+1}).
\end{equation*}
Note that $(u^1,v^1)=(u^0,v^0)$ on $I_{\ell_0}\times \Om_{\ell_0}$. Clearly $(u^1,v^1)$ satisfies 
\begin{equation*}
\int_{\Om_{\ell_0+1}} D\nabla v^1\cdot \nabla \varphi ~dx +  \int_{I_{\ell_0+1}} \nu v^1(x_1,0)\varphi  ~ dx_1 =  \int_{\Om_{\ell_0+1}} f(v^1)\varphi ~dx + \int_{I_{\ell_0+1}} \mu u^1 \varphi ~dx_1 ~~\forall \varphi \in V_{\ell_0+1},
\end{equation*}
\begin{equation*}
\int_{I_{\ell_0+1}} D'{u^1}'\psi'  +\mu u^1 \psi~dx_1 = \int_{I_{\ell_0+1}} \nu v^1(x_1,0)\psi ~dx_1 +  \int_{I_{\ell_0+1}} g(u^1) \psi ~dx_1 ~~\forall \psi  \in H^1_0(I_{\ell_0+1}).
\end{equation*}
By induction one constructs a sequence $(u_{n,k},v_{n,k})$ extracted from the preceding converging toward $(u^k,v^k)$ satisfying the equations above where we have replaced $\ell_0 + 1$ by  $\ell_0 + k$. 
Then using the usual diagonal process it is clear that the sequence $(u_{n,n},v_{n,n})$ will converge toward a solution to \eqref{5.1} nontrivial thanks to \eqref{5.3}. This completes the proof of the theorem. 
\end{proof}

{\bf Acknowledgement} :  Part of this work was performed when the first author was visiting USTC in Hefei. We thank the university and the school of mathematical sciences for their hospitality and their support.

\end{document}